\documentclass[11pt]{amsart}

\usepackage{latexsym}
\usepackage{amssymb}
\usepackage{amsmath}
\usepackage{mathrsfs}

\newtheorem{theorem}{Theorem}[section]
\newtheorem{lemma}[theorem]{Lemma}
\newtheorem{proposition}[theorem]{Proposition}
\newtheorem{corollary}[theorem]{Corollary}

\theoremstyle{definition}
\newtheorem{definition}[theorem]{Definition}

\theoremstyle{remark}
\newtheorem{remark}[theorem]{Remark}

\newcommand{\F}{\mathbb{F}}
\newcommand{\N}{\mathbb{N}}

\newcommand{\Q}{\mathbb{Q}}
\newcommand{\R}{\mathbb{R}}

\newcommand{\G}{\mathcal{G}}
\newcommand{\p}{\mathcal{P}}
\newcommand{\U}{\mathcal{U}}

\newcommand{\pfin}{\text{Fin}}

\newcommand{\num}{\mathfrak{n}}

\newcommand{\ring}{\mathfrak}
\newcommand{\A}{\mathfrak{A}}
\newcommand{\B}{\mathfrak{B}}
\newcommand{\C}{\mathfrak{C}}

\newcommand{\sh}{\text{sh}}

\begin{document}

\title{Elementary numerosity and measures}

\author{Vieri Benci}
\author{Emanuele Bottazzi}
\author{Mauro Di Nasso}

\address{Dipartimento di Matematica,
Universit\`{a} di Pisa, Italy and Department of Mathematics, College of Science, King Saud University, Riyadh,  Saudi Arabia}
\email{benci@dma.unipi.it}

\address{Dipartimento di Matematica,
Universit\`{a} di Trento, Italy.}
\email{Emanuele.Bottazzi@unitn.it}

\address{Dipartimento di Matematica,
Universit\`{a} di Pisa, Italy.}
\email{dinasso@dm.unipi.it}


\maketitle

\begin{abstract}
In this paper we introduce the notion of elementary numerosity
as a special function defined on all subsets of a given set $X$
which takes values in a suitable non-Archimedean field, and
satisfies the same formal properties of finite cardinality.
We investigate the relationships between this notion and
the notion of measure. The main result is that every
non-atomic finitely additive measure is obtained from a suitable
elementary numerosity by simply taking its ratio to a unit.
In the last section we give applications to this result.
\end{abstract}

\tableofcontents

\section*{Introduction}

In mathematics there are essentially two main ways to
estimate the size of a set, depending on whether
one is working in a discrete or in a continuous setting.

In the continuous case, one uses the notion
of (finitely additive) measure, namely a function
$m$ taking real values
and which satisfies the following properties:
\begin{enumerate}
\item
$m(\emptyset)=0$
\item
$m(A)\ge 0$
\item
$m(A\cup B)=m(A)+m(B)$ whenever $A\cap B=\emptyset$
\end{enumerate}

In the discrete case, one uses
the notion of cardinality $\frak{n}$ that strengthens
the three properties
itemized above as follows:

\begin{enumerate}
\item[($\frak{n}.1$)]
$\frak{n}(\emptyset)=0$
\item[($\frak{n}.2$)]
$\frak{n}(A)\ge 0$
\item[($\frak{n}.3$)]
$\frak{n}(A\cup B)=\frak{n}(A)+\frak{n}(B)$ whenever $A\cap B=\emptyset$
\item[($\frak{n}.4$)]
$\frak{n}\left(\{x\}\right)=1$ for all singletons
\end{enumerate}

\smallskip
The goal of this paper is to investigate the relationships
between these two notions. Of course, this problem is interesting
when the sets are infinite. Remark that the theory of infinite
cardinality is not adequate to this end; for example,
all sets of reals with positive Lebesgue measure
have the same cardinality.
On the contrary, the notion of numerosity,
first introduced in \cite{be95,bd},
gives a coherent way of extending finite cardinalities
and their main properties to infinite sets.

\smallskip
In this paper we introduce the related concept of
\emph{elementary numerosity} as a special function
defined on \emph{all} subsets of a given set $X$ that takes
values into a suitable ordered field $\mathbb{F}$ and satisfies
the four properties of finite cardinalities itemized above.
Remark that if $X$ is infinite, then the range
of such a function $\frak{n}$ necessarily contains infinite numbers,
and hence the field $\mathbb{F}$ must be non-Archimedean.
Notice that also Cantorian cardinality satisfies properties
$(\frak{n}.1),(\frak{n}.2),(\frak{n}.3),(\frak{n}.4)$;
the fundamental difference is that ``numerosities" are
required to be elements of a field.

\smallskip
By taking ratios to a fixed ``measure unit" $\beta>0$,
one has a canonical way of getting a real-valued
finitely additive measure. This construction turns
out to be really general. In fact, the main result of this paper
shows that every finitely additive non-atomic measure can be
obtained in this way. Namely, we shall prove the following:

\medskip
\noindent
\textbf{Theorem.}
Let $(\Omega,\A,\mu)$ be a non-atomic finitely additive measure.
Then there exist

\smallskip
\begin{itemize}
\item
a non-Archimedean field $\F\supset\R$\,;

\smallskip
\item
an elementary numerosity $\num:\p(\Omega)\to[0,+\infty)_\F$\,;

\smallskip
\item
a positive number $\beta\in\F$
\end{itemize}

\smallskip
\noindent
such that
$$\mu(A)\ =\ \sh\left(\frac{\num(A)}{\beta}\right)\ \text{for all }A\in\A.\footnote
{~$\sh(\xi)$ is the unique real number which is
infinitely close to $\xi$ (see Section \ref{sec-elementary}).}$$


\smallskip
The last part of the paper is devoted to
selected applications: the first one is about Lebesgue measure,
and the second one is about non-Archimedean probability.
In particular, following ideas from \cite{nap}, we consider a
model for infinite sequences
of coin tosses which is coherent with Laplace original view.
Indeed, probability of an event is defined as
the numerosity of positive outcomes divided by the
numerosity of all possible outcomes; moreover, the probability of
cylindrical sets exactly coincides
with the usual Kolmogorov probability.

\bigskip
\section{Terminology and preliminary notions}

We fix here our terminology, and recall a few basic facts from measure
theory that will be used in the sequel.

\smallskip
Let us first agree on notation.
We write $A\subseteq B$ to mean that $A$ is a subset of $B$,
and we write $A\subset B$ to mean that $A$ is a \emph{proper} subset of $B$.
The \emph{complement} of a set $A$ is denoted by $A^c$, and its \emph{powerset}
is denoted by $\p(A)$. We write $A_1\sqcup\ldots\sqcup A_n$ to denote a
\emph{disjoint union}.
By $\N$ we denote the set of positive integers, and by
$\N_0=\N\cup\{0\}$ the set of non-negative integers.
For an ordered field $\F$, we denote by
$[0,\infty)_\F=\{x\in\F\mid x\ge 0\}$
the set of its non-negative elements.
We shall write $[0,+\infty]_\R$ to denote the set of non-negative
real numbers plus the symbol $+\infty$, and we agree
that $x+\infty=+\infty+x=+\infty+\infty=+\infty$ for all $x\in\R$.

\medskip
\begin{definition}
A \emph{finitely additive measure} is a triple $(\Omega,\A,\mu)$ where:

\smallskip
\begin{itemize}
\item
The \emph{space} $\Omega$ is a non-empty set;

\smallskip
\item
$\A$ is a \emph{ring of sets} over $\Omega$, \emph{i.e.}
a non-empty family of subsets of $\Omega$ which is closed under
finite unions and intersections, and under relative complements, \emph{i.e.}
$A,B\in\A\Rightarrow A\cup B, A\cap B, A\setminus B\in\A$;\footnote
{~Actually, the closure under intersections follow from the other two properties,
since $A\cap B=A\setminus(A\setminus B)$.}

\smallskip
\item
$\mu:\A\to [0,+\infty]_\R$ is an \emph{additive function}, \emph{i.e.}
$\mu(A\cup B)=\mu(A)+\mu(B)$ whenever $A,B\in\A$ are disjoint.\footnote
{~Such functions $\mu$ are sometimes called \emph{contents} in the literature.}
We also assume that $\mu(\emptyset)=0$.
\end{itemize}

\smallskip
The measure $(\Omega,\A,\mu)$ is called \emph{non-atomic} when all
finite sets in $\A$ have measure zero. We say that
$(\Omega,\A,\mu)$ is a \emph{probability} when $\mu:\A\to[0,1]_\R$
takes values in the unit interval.

\end{definition}

\medskip
For simplicity, in the following we shall often identify the
triple $(\Omega,\A,\mu)$ with the function $\mu$.

\smallskip
Remark that a finitely additive measure $\mu$ is
necessarily \emph{monotone}, \emph{i.e.}

\smallskip
\begin{itemize}
\item
$\mu(A)\le\mu(B)$ for all $A,B\in\A$ with $A\subseteq B$.
\end{itemize}


\medskip
\begin{definition}
A finitely additive measure $\mu$ defined on a ring of sets $\A$ is
called a \emph{pre-measure} if it is $\sigma$-\emph{additive},
\emph{i.e.} if for every countable family $\{A_n\}_{n\in\N}\subseteq\A$
of pairwise disjoint sets whose union lies in $\A$, it holds:
$$\mu\left(\bigsqcup_{n\in\N}A_n\right)\ =\ \sum_{n=1}^\infty\mu(A_n).$$
A \emph{measure} is a pre-measure which is defined on a $\sigma$-\emph{algebra},
\emph{i.e.} on a ring of sets which is closed under countable unions and intersections.
\end{definition}

\medskip
\begin{definition}
An \emph{outer measure} on a set $\Omega$ is a function
$$M:\p(\Omega)\to[0,+\infty]_\R$$
defined on all subsets of $\Omega$ which is \emph{monotone}
and $\sigma$-\emph{subadditive}, \emph{i.e.}
$$M\left(\bigcup_{n\in\N}A_n\right)\ \leq\ \sum_{n\in\N}M(A_n).$$
It is also assumed that $M(\emptyset)=0$.
\end{definition}

\medskip
\begin{definition}
Given an outer measure $M$ on $\Omega$, the following family is
called the \emph{Caratheodory $\sigma$-algebra} associated to $M$:
$$\mathfrak{C}_M\ =\ \left\{X\subseteq\Omega\mid
M(X)=M(X\cap Y)+M(X\setminus Y)\ \text{for all }Y\subseteq\Omega\right\}.$$
\end{definition}

\medskip
A well known theorem of Caratheodory states that the above family is
actually a $\sigma$-algebra, and that the restriction of $M$ to
$\mathfrak{C}_M$ is a \emph{complete} measure, \emph{i.e.}
a measure where $M(X)=0$ implies $Y\in\mathfrak{C}_M$ for all $Y\subseteq X$.
This result is usually combined with the property that every pre-measure
$\mu$ over a ring $\A$ of subsets of $\Omega$
is canonically extended to the outer measure
$\overline{\mu}:\p(\Omega)\to[0,\infty]_\R$ defined by putting:
$$\overline{\mu}(X)\ =\
\inf\left\{\sum_{n=1}^\infty\mu(A_n)\,\Big|\,
\{A_n\}_n\subseteq\A\ \&\ X\subseteq\bigcup_{n\in\N}A_n\right\}.$$

Indeed, a fundamental result in measure theory is that
the above function $\overline{\mu}$ is actually an outer measure
that extends $\mu$, and that the associated Caratheodory $\sigma$-algebra
$\mathfrak{C}_{\overline{\mu}}$ includes $\A$.
Moreover, such an outer measure $\overline{\mu}$ is \emph{regular},
\emph{i.e.} for all $X\in\p(\Omega)$ there exists $C\in\mathfrak{C}_{\overline{\mu}}$
such that $X\subseteq C$ and $\overline{\mu}(X)=\overline{\mu}(C)$.
(See \emph{e.g}. \cite{Yeh} Prop. 20.9.)

\smallskip
In the proof of our main theorem, we shall use
an ultrapower $\R^I/\U$ of the real numbers modulo
a suitable ultrafilter. Recall that
an \emph{ultrafilter} $\U$ on a set $I$ is a maximal
family of subsets of $I$ which has the finite
intersection property (FIP): $A_1\cap\ldots\cap A_n\ne\emptyset$
for any choice of finitely many $A_i\in\U$.
Equivalently, $\U$ is a family of non-empty subsets of $I$ that is
closed under supersets, finite intersections, and satisfies
the property $A\notin\U\Rightarrow I\setminus A\in\U$.
Remark that an ultrafilter $\U$ can also be characterized as
a family of sets that have measure $1$
with respect to a suitable finitely additive $\{0,1\}$-valued
measure $\mu:\p(I)\to\{0,1\}$.
By applying Zorn's lemma, it is shown that every family
$\F\subseteq\p(I)$ with the FIP can be extended to an ultrafilter on $I$.

\smallskip
The \emph{ultrapower} $\F=\R^I/\U$ of the real numbers modulo the
ultrafilter $\U$ is the ordered field where:

\begin{itemize}
\item
Elements of $\F$ are the real $I$-sequences $\langle\sigma\rangle_\U$ defined
$\U$-almost everywhere,
\emph{i.e.} $\langle\sigma\rangle_\U=\langle\tau\rangle_\U$ when
$\{i\in I\mid\sigma(i)=\tau(i)\}\in\U$\,;
\item
The order relation and the sum and product operations
are defined point-wise, \emph{i.e.} $\langle\sigma\rangle_\U<\langle\tau\rangle_\U$
if $\sigma(i)<\tau(i)$
$\U$-almost everywhere, $\sigma+\tau=\zeta$ if $\sigma(i)+\tau(i)=\zeta(i)$
$\U$-almost everywhere, and similarly for the product.
\end{itemize}

\smallskip
For detailed information about ultrafilters and the general construction
of ultrapower, the reader is referred to \emph{e.g.} \cite{ck}.

\bigskip
\section{Elementary numerosity}\label{sec-elementary}

Inspired by the idea of numerosity, we now aim at refining the notion
of finitely additive measure in such a way that also single points count.
To this end, one needs to consider \emph{superreal fields} $\F\supseteq\R$,
\emph{i.e.} ordered fields which extend the real line.

\smallskip
Remark that if the field $\F\supset\R$ is a proper extension, then $\F$ is necessarily
\emph{non-Archimedean}, \emph{i.e.} it contains \emph{infinitesimal
numbers} $\epsilon\ne 0$ such that $-1/n<\epsilon<1/n$ for all $n\in\N$.
We say that two elements $\xi,\zeta\in\F$ are \emph{infinitely close}, and
write $\xi\approx\zeta$, when their difference $\xi-\zeta$ is infinitesimal.
A number $\xi\in\F$ is called \emph{finite} when $-n<\xi<n$ for some $n\in\N$, and it is called
\emph{infinite} otherwise. Clearly, a number $\xi$ is infinite if and only if
its reciprocal $1/\xi$ is infinitesimal.
Since $\F\supseteq\R$, by the completeness property of the real line
it is easily verified that every finite $\xi\in\F$ is infinitely close to a
unique real number $r$ (just take $r=\inf\{x\in\R\mid x>\xi\}$).
Such a number $r$ is called the \emph{shadow} (or \emph{standard part}) of $\xi$,
and notation $r=\sh(\xi)$ is used. Notice that $\sh(\xi+\zeta)=\sh(\xi)+\sh(\zeta)$
and $\sh(\xi\cdot\zeta)=\sh(\xi)\cdot\sh(\zeta)$ for all finite $\xi,\zeta$.
By abusing notation, we shall write $\sh(\xi)=+\infty$ when
$\xi$ is infinite and positive, and $\sh(\xi)=-\infty$
when $\xi$ is infinite and negative.

\medskip
\begin{definition}
An \emph{elementary numerosity} on a set $\Omega$ is a function
$$\num:\p(\Omega)\to [0,+\infty)_\F$$
defined for all subsets of $\Omega$,
taking values into the non-negative part of a superreal field $\F$,
and such that the following two conditions are satisfied:

\smallskip
\begin{enumerate}
\item
$\num(\{x\})=1$ for every point $x\in\Omega$\,;

\smallskip
\item
$\num(A\cup B)=\num(A)+\num(B)$ whenever $A$ and $B$ are disjoint.
\end{enumerate}
\end{definition}

\medskip
As straight consequences of the definition, we obtain
that elementary numerosities can be seen as generalizations of
finite cardinalities.

\medskip
\begin{proposition}
Let $\num$ be an elementary numerosity. Then:

\smallskip
\begin{enumerate}
\item
$\num(A)=0$ if and only if $A=\emptyset$;

\smallskip
\item
If $A\subset B$ is a proper subset, then $\num(A)<\num(B)$.

\smallskip
\item
If $F$ is a finite set of cardinality $n$, then $\num(F)=n$;
\end{enumerate}
\end{proposition}

\begin{proof}
Notice that $\num(\emptyset)=\num(\emptyset\cup\emptyset)=
\num(\emptyset)+\num(\emptyset)$,
and $x=0$ is the only number $x\in\F$ such that $x+x=x$.
If $A\subseteq B$ then $\mu(B)=\mu(A)+\mu(B\setminus A)\ge\mu(A)$.
Moreover, if $A\subset B$ is a proper subset and $x\in B\setminus A$,
then $\mu(B)\ge\mu(A\cup\{x\})=\mu(A)+\mu(\{x\})=\mu(A)+1>\mu(A)$.
In consequence, $\mu(A)>0$ for all non-empty sets $A$.
Finally, the last property directly follows by additivity and the
fact that every singleton has measure $1$.
\end{proof}

\smallskip
Remark that if one takes $\F=\R$ then elementary numerosities $\num$ exist
on a set $\Omega$ if and only if $\Omega$ is finite and
$\num$ is the finite cardinality. However, we shall see that
our assumption which allows $\num$ to take non-Archimedean values,
will make it possible to extend the ``counting measure"
as given by finite cardinality, to arbitrary infinite sets.

\begin{proposition}
Let $\num:\p(\Omega)\to[0,+\infty)_\F$ be an elementary numerosity,
and for every $\beta>0$ in $\F$ define the function
$\num_\beta:\p(\Omega)\to[0,+\infty]_\R$ by posing
$$\num_\beta(A)\ =\ \sh\left(\frac{\num(A)}{\beta}\right).$$
Then $\num_\beta$ is a finitely additive measure defined for
\emph{all} subsets of $\Omega$. Moreover,
$\num_\beta$ is non-atomic if and only if $\beta$ is an infinite number.
\end{proposition}

\begin{proof}
For all disjoint $A,B\subseteq\Omega$, one has:
\begin{eqnarray*}
\nonumber
\num_\beta(A\cup B) &=&
\sh\left(\frac{\num(A\cup B)}{\beta}\right)\ =\
\sh\left(\frac{\num(A)}{\beta}+\frac{\num(B)}{\beta}\right)
\\
\nonumber
{} &=&
\sh\left(\frac{\num(A)}{\beta}\right)+\sh\left(\frac{\num(B)}{\beta}\right)\ =\
\num_\beta(A)+\num_\beta(B).
\end{eqnarray*}

Notice that the measure $\num_\beta$ is non-atomic if and only if
$\num_\beta(\{x\})=\sh(1/\beta)=0$, and this
holds if and only if $\beta$ is infinite.
\end{proof}


\medskip
The above class of measures turns out to be really general.
In the next section we shall show that every finitely
additive non-atomic measure is in fact a restriction of a
suitable $\num_\beta$.

\bigskip
\section{The main result}

\begin{theorem} \label{main theorem}
Let $(\Omega,\A,\mu)$ be a non-atomic finitely additive measure.
Then there exist

\smallskip
\begin{itemize}
\item
a non-Archimedean field $\F\supseteq\R$\,;

\smallskip
\item
an elementary numerosity $\num:\p(\Omega)\to[0,+\infty)_\F$\,;
\end{itemize}

\smallskip
\noindent
such that for every positive number of the form
$\beta=\frac{\num(A^*)}{\mu(A^*)}$ one has
$$\mu(A)\ =\ \num_\beta(A)\ \text{for all }A\in\A.$$

Moreover, if $\B\subseteq\A$ is a subring whose non-empty
sets have all positive measure, then we can also assume that
$$\num(B)\ =\ \num(B')\ \ \text{for all }B,B'\in\B\ \text{such that }
\mu(B)=\mu(B').$$
\end{theorem}

\begin{proof}
Denote by $\A_f$ (by $\B_f$) the set of all elements of $\A$
(of $\B$, respectively) which have finite measure.
Let $\Lambda=\pfin(\Omega)$ be the family of all finite subsets of $\Omega$,
and define the following sets.

\smallskip
\begin{itemize}
\item
For all $x\in\Omega$, let
$$\widehat{x}\ =\ \left\{\lambda\in\Lambda : x\in\lambda\right\}.$$

\smallskip
\item
For all $A,A'\in\A_f$ with $\mu(A')>0$ and for all $n\in\N$, let
$$\Gamma(A,A',n)\ = \ \left\{\lambda\in\Lambda :
\left|\frac{|\lambda\cap A|}{|\lambda\cap A'|} - \frac{\mu(A)}{\mu(A')}\right|\ <\
\frac{1}{n}\right\}.$$

\smallskip
\item
For all non-empty $B,B'\in\B_f$, let
$$\Theta(B,B')\ =\ \left\{\lambda\in\Lambda : |B\cap\lambda|=|B'\cap\lambda|\right\}.$$
\end{itemize}

\smallskip
Then consider the following family of subsets of $\Lambda$:
\begin{eqnarray*}
\nonumber
\G & = &
\left\{\widehat{x}\mid x\in\Omega\right\}\ \bigcup\
\left\{\Gamma(A,A',n)\mid A,A'\in\A_f,\ \mu(A')>0,\ n\in\N\right\}
\\
\nonumber
{} & {} &
 \bigcup\left\{\Theta(B,B')\mid B,B'\in\B,\ B,B'\ne\emptyset\right\}.
\end{eqnarray*}

\smallskip
We want to show that all finite intersections of elements of $\G$ are non-empty.
To this end, we shall use the following combinatorial result,
whose proof is put off to the Appendix.

\medskip
\begin{lemma}\label{combinatorial}
Let $(\Omega,\A,\mu)$ be a non-atomic finitely additive measure, and let
$\B\subseteq\A$ be a subring of subsets of $\Omega$
whose non-empty sets have all positive finite measure.
Given $m\in\N$, given finitely many points $x_1,\ldots,x_k\in\Omega$,
and given finitely many non-empty sets $A_1,\ldots,A_n\in\A$ having finite measure,
there exists a finite subset $\lambda\subset\Omega$ that satisfies the following properties:

\smallskip
\begin{enumerate}
\item
$x_1,\ldots,x_k\in\lambda$\,;

\smallskip
\item
If $A_i,A_j\in\B$ are such that $\mu(A_i)=\mu(A_j)$
then $|\lambda\cap A_i|=|\lambda\cap A_j|$\,;
	
\smallskip
\item
If $\mu(A_j)\ne 0$ then for all $i$:
$$\left|\frac{|\lambda\cap A_i|}{|\lambda\cap A_j|} -
\frac{\mu(A_i)}{\mu(A_j)}\right|\ <\ \frac{1}{m}.$$
\end{enumerate}
\end{lemma}

\medskip
\noindent
Now let finitely many elements of $\G$ be given, say
$$\widehat{x}_1,\ldots,\widehat{x}_k;\
\Gamma(A_1,A'_1,n_1),\ldots,\Gamma(A_u,A'_u,n_u);\
\Theta(B_1,B'_1),\ldots,\Theta(B_v,B'_v)\,.$$
Pick $m=\max\{n_1,\ldots,n_u\}$ and apply the above Lemma
to get the existence a finite set $\lambda\in\Lambda$ such that

\begin{enumerate}
\item
$x_1,\ldots,x_k\in\lambda$\,;

\smallskip
\item
For all $i=1,\ldots,v$, if $\mu(B_i)=\mu(B'_i)$
then $|\lambda\cap B_i|=|\lambda\cap B'_i|$\,;
	
\smallskip
\item
For all $i,j=1,\ldots, u$, if $\mu(A_j)\ne 0$ then
$$\left|\frac{|\lambda\cap A_i|}{|\lambda\cap A_j|} -
\frac{\mu(A_i)}{\mu(A_j)}\right|\ <\ \frac{1}{m}.$$
\end{enumerate}

\smallskip
Then it is readily verified that such a $\lambda$
belongs to all considered sets of $\G$.
In consequence of this \emph{finite intersection property},
the family $\G\subset\p(\Lambda)$ can be extended to an ultrafilter
$\U$ on $\Lambda$.
Now define the following:

\smallskip
\begin{itemize}
\item
$\F=\R^\Lambda/\U$ is the ordered field obtained as the
\emph{ultrapower} of $\R$ modulo the ultrafilter $\U$.
We identify each real number $r$ with the corresponding
constant sequence $\langle c_r\rangle_\U$
defined $\U$-almost everywhere, so that $\R\subseteq\F$.

\smallskip
\item
$\num:\p(\Omega)\to[0,+\infty)_\F$ is the function where
$$\num(X)\ =\ \langle |X\cap\lambda|:\lambda\in\Lambda\rangle_\U$$
is the $\U$-equivalence class of the $\Lambda$-sequence of natural
numbers obtained by taking the number of elements
of $X$ found in every finite subset of $\Omega$.

\end{itemize}

\smallskip
Let us now verify that all the desired properties are satisfied.
Given $A,A'\in\A_f$ with $\mu(A')\ne 0$, for every $n\in\N$ we have that
$$\left\{\lambda\in\Lambda :
\left|\frac{|\lambda\cap A|}{|\lambda\cap A'|} - \frac{\mu(A)}{\mu(A')}\right|\ <\
\frac{1}{n}\right\} =\ \Gamma(A,A',n)\ \in\ \G\ \subset\U\,,$$
and so
$$\left|\frac{\num(A)}{\num(A')} - \frac{\mu(A)}{\mu(A')}\right|\ <\ \frac{1}{n}.$$
As this holds for every $n$, we conclude that
$$\sh\left(\frac{\num(A)}{\num(A')}\right)\ =\ \frac{\mu(A)}{\mu(A')}.$$
In consequence, for every positive number $\beta\in\F$ of the form
$\frac{\num(A^*)}{\mu(A^*)}$, one has
\begin{eqnarray*}
\nonumber
\num_\beta(A) & = &
\sh\left(\frac{\num(A)}{\beta}\right)\ =\
\sh\left(\frac{\num(A)}{\num(A^*)}\cdot\mu(A^*)\right)
\\
\nonumber
{} & = & \sh\left(\frac{\num(A)}{\num(A^*)}\right)\cdot\mu(A^*)\ =\
\frac{\mu(A)}{\mu(A^*)}\cdot\mu(A^*)\ =\ \mu(A).
\end{eqnarray*}

\smallskip
As for property $(2)$, if $B,B'\in\B_f$ are non-empty sets with
$\mu(B)=\mu(B')$, then
$$\{\lambda\in\Lambda : |\lambda\cap B|=|\lambda\cap B'|\}\ =\
\Theta(B,B')\ \in\ \G\ \subset\ \U\,,$$
and hence $\num(B)=\num(B')$.
\end{proof}


\medskip
As a straight consequence, we obtain the following

\medskip
\begin{theorem} \label{inner measure theorem}
Let $\A$ be a ring of subsets of $\Omega$ and let
$\mu:\A\to[0,+\infty]_\R$ be a non-atomic pre-measure.
Then, along with the associated outer measure $\overline{\mu}$,
there exists an ``inner'' finitely additive measure
$$\underline{\mu}:\p(\Omega)\to[0,+\infty]_\R$$
such that:

\smallskip
\begin{enumerate}
\item
There exists an elementary numerosity $\num:\p(\Omega)\to\F$
such that $\underline{\mu}=\num_\beta$
for every positive number of the form
$\beta=\frac{\num(A^*)}{\mu(A^*)}$.

\smallskip
\item
$\underline{\mu}(C)=\overline{\mu}(C)$ for all
$C\in\mathfrak{C}_\mu$, the Caratheodory $\sigma$-algebra associated to $\mu$.
In particular, $\underline{\mu}(A)=\mu(A)=\overline{\mu}(A)$ for all
$A\in\mathfrak{A}$.

\smallskip
\item
$\underline{\mu}(X)\le\overline{\mu}(X)$ for all
$X\subseteq\Omega$.
\end{enumerate}
\end{theorem}

\begin{proof}
By \emph{Caratheodory extension theorem}, the restriction
of $\overline{\mu}$ to $\mathfrak{C}_\mu$ is a measure
that agrees with $\mu$ on $\A$.
By applying the previous theorem to $\overline{\mu}_{|\mathfrak{C}_\mu}$,
we obtain the existence of
an elementary numerosity
$\num:\p(\Omega)\to[0,+\infty)_\F$
such that for every positive number of the form
$\beta=\frac{\num(A^*)}{\mu(A^*)}$ one has
$\num_\beta(C)=\overline{\mu}(C)$ for all
$C\in\mathfrak{C}_\mu$. We claim that
$\underline{\mu}=\num_\beta:\p(\Omega)\to[0,+\infty]_\R$
is the desired ``inner" finitely additive measure.

\smallskip
Properties $(1)$ and $(2)$ are trivially satisfied by our definition
of $\underline{\mu}$, so we are left to show $(3)$.
For every $X\subseteq\Omega$, by definition of outer measure we have that
for every $\epsilon>0$
there exists a countable union $A=\bigcup_{n=1}^\infty A_n$
of sets $A_n\in\A$ such that $A\supseteq X$ and
$\sum_{n=1}^\infty\mu(A_n)\le\overline{\mu}(X)+\epsilon$.
Notice that $A$ belongs to the $\sigma$-algebra generated by $\A$,
and hence $A\in\mathfrak{C}_\mu$. In consequence,
$\underline{\mu}(A)=\num_\beta(A)=\overline{\mu}(A)$.
Finally, by monotonicity of the finitely additive measure $\underline{\mu}$,
and by $\sigma$-subadditivity of the outer measure $\overline{\mu}$, we obtain:
$$\underline{\mu}(X)\ \le\ \underline{\mu}(A)\ =\ \overline{\mu}(A)\ \le\
\sum_{n=1}^\infty\overline{\mu}(A_n)\ =\ \sum_{n=1}^\infty\mu(A_n)\ \le\
\overline{\mu}(X)+\epsilon.$$
As $\epsilon>0$ is arbitrary, the desired inequality
$\underline{\mu}(X)\le\overline{\mu}(X)$ follows.
\end{proof}

\medskip
It seems of some interest to investigate the properties of the
estension of the Caratheodory algebra given by family of all sets
for which the outer measure coincides with the above ``inner measure":
$$\mathfrak{C}(\num_\beta)\ =\ \left\{X\subseteq\Omega\mid\underline{\mu}(X)=
\overline{\mu}(X)\right\}.$$
Clearly, the properties of $\mathfrak{C}(\num_\beta)$
may depend on the choice of the
elementary numerosity $\num_\beta$.

\smallskip
Theorem \ref{inner measure theorem} ensures that
the inclusion $\C_\mu \subseteq \C(\num_\beta)$ always holds.
Moreover, this inclusion is an equality if and only if all $X \not \in \C_\mu$ 
satisfy the inequality $\underline{\mu}(X) < \overline{\mu}(X)$.
It turns out that this property is equivalent to a number of other 
statements.

\medskip
\begin{proposition}
The following are equivalent:

\smallskip
\begin{enumerate}
\item 
$X\not\in \C_\mu \Rightarrow \underline{\mu}(X) < \overline{\mu}(X)$ 
and $\underline{\mu}(X^c) < \overline{\mu}(X^c)$.

\smallskip
\item
$\underline{\mu}(X)=\overline{\mu}(X)\Longleftrightarrow\underline{\mu}(X^c)=
\overline{\mu}(X^c)$.

\smallskip
\item 
$\underline{\mu}(X) = 0 \Longleftrightarrow \overline{\mu}(X) = 0$.
\end{enumerate}
\end{proposition}

\begin{proof}
$(1)\Rightarrow(2)$.
Suppose towards a contradiction that $(1)$ holds but $(2)$ is false.
The latter hypothesis ensures the existence of a set $X \subseteq \R$ 
such that $\underline{\mu}(X)=\overline{\mu}(X)$ and 
$\underline{\mu}(X^c)<\overline{\mu}(X^c)$.
Thanks to Theorem \ref{inner measure theorem}, we deduce that 
$X\not\in\C_\mu$; at this point, by $(1)$ we get the contradiction 
$\underline{\mu}(X)<\overline{\mu}(X)$.

\smallskip
$(2)\Rightarrow(3)$.
The implication $\overline{\mu}(X)=0\Rightarrow\underline{\mu}(X)=0$ 
is always true.
On the other hand, if $\underline{\mu}(X)=0$, then 
$\underline{\mu}(X^c)=\underline{\mu}(\Omega)=\overline{\mu}(\Omega)$.
By the inequality $\underline{\mu}(X^c)\leq\overline{\mu}(X^c)$, 
we deduce $\overline{\mu}(X^c)=\overline{\mu}(\Omega)=\underline{\mu}(X^c)$ and, 
thanks to $(2)$, also $\overline{\mu}(X)=0$ follows.

\smallskip
$(3)\Rightarrow(1)$.
Suppose towards a contradiction that $(3)$ holds but $(1)$ is false.
The latter hypothesis ensures the existence of a set $X\not\in\C_\mu$ 
satisfying $\underline{\mu}(X)=\overline{\mu}(X)$ and 
$\underline{\mu}(X^c)<\overline{\mu}(X^c)$.
Thanks to Propositions 20.9 and 20.11 of \cite{Yeh},
we can find a set $A\in\mathfrak{C}_\mu$ satisfying 
$A\supset X$, $\overline{\mu}(A)=\overline{\mu}(X)$ and 
$\overline{\mu}(A\setminus X)> 0$.
From the hypothesis $\underline{\mu}(X)=\overline{\mu}(X)$ we 
obtain the following equalities:
$$\underline{\mu}(X)\ =\ \overline{\mu}(X)\ =\ 
\overline{\mu}(A)\ =\ \underline{\mu}(A)$$
which imply $\underline{\mu}(A\setminus X)=0$.
Finally, by the hypothesis $(3)$, we obtain the contradiction 
$\overline{\mu}(A\setminus X)=0$.
\end{proof}

\bigskip
\section{Applications}

In this last section, we present two consequences
of Theorems \ref{main theorem} and \ref{inner measure theorem} that
may have some relevance in applications.

\medskip
\subsection{Elementary numerosities and Lebesgue measure}

The first application that we show is about the existence of
an elementary numerosity which is consistent with Lebesgue measure.

\medskip
\begin{corollary}
Let $(\R,\mathfrak{L},\mu_L)$ be the \emph{Lebesgue measure} over $\R$.
Then there exists an elementary numerosity
$\num:\p(\R)\rightarrow\F$ such that:

\smallskip
\begin{enumerate}
\item
$\num([x,x+a))=\num([y,y+a))$ for all $x,y\in\R$ and for all $a>0$.

\smallskip
\item
$\num([x,x+a))=a\cdot\num([0,1))$ for all rational numbers $a>0$.


\smallskip
\item
$\sh\left(\frac{\num(X)}{\num([0,1))}\right)=\mu_L(X)$ for all $X\in\mathfrak{L}$.

\smallskip
\item
$\sh\left(\frac{\num(X)}{\num([0,1))}\right)\le
\overline{\mu}_L(X)$ for all $X\subseteq\R$.
\end{enumerate}
\end{corollary}

\begin{proof}
Notice that the family of half-open intervals
$$\ring{I}\ =\ \left\{[x,x+a)\mid x\in\R \ \&\ a>0\right\}$$
generates a subring $\B\subset\ring{L}$ whose non-empty sets have
all positive measure.
Then by combining Theorems \ref{main theorem} and \ref{inner measure theorem}
we obtain the existence of an elementary numerosity $\num:\p(\R)\to\F$
such that for every positive number of the form
$\beta=\frac{\num(A^*)}{\mu(A^*)}$ one has:

\smallskip
\begin{enumerate}
\item[$(i)$]
$\num(X)=\num(Y)$ whenever $X,Y\in\B$ are such that
$\mu_L(X)=\mu_L(Y)$\,;

\smallskip
\item[$(ii)$]
$\num_\beta(X)=\mu_L(X)$ for all $X\in\mathfrak{L}$\,;

\smallskip
\item[$(iii)$]
$\num_\beta(X)\le\overline{\mu}_L(X)$ for all $X\subseteq\R$.
\end{enumerate}

\smallskip
Since $[x,x+a)\in\B$ for all $x\in\R$ and for all $a>0$, property $(1)$
directly follows from $(i)$. In order to prove $(2)$,
it is enough to show that $\num([0,a))=a\cdot\num([0,1)$ for all positive
$a\in\Q$. Given $p,q\in\N$, by $(1)$ and additivity we have that
$$\num\left(\left[0,\frac{p}{q}\right)\right)\ =\
\num\left(\bigsqcup_{i=0}^{p-1}\left[\frac{i}{q},\frac{i+1}{q}\right)\right)\ =\
\sum_{i=0}^{p-1}\num\left(\left[\frac{i}{q},\frac{i+1}{q}\right)\right)\ =\
p\,\cdot\num\left(\left[0,\frac{1}{q}\right)\right).$$
In particular, for $p=q$ we get that $\num([0,1))=q\cdot\num([0,1/q))$,
and hence property $(2)$ follows:
$$\num\left(\left[0,\frac{p}{q}\right)\right)\ =\
\frac{p}{q}\cdot\num\left([0,1)\right).$$

Finally, if we take as
$$\beta\ =\ \frac{\num([0,1))}{\mu_L([0,1))}\ =\ \num([0,1)),$$
then $(ii)$ and $(iii)$ correspond to properties
$(3)$ and $(4)$, respectively.
\end{proof}

\medskip
\begin{remark}
Notice that every \emph{non}-Lebesgue measurable set $X$ such that
$$\num_\beta(X)\ =\ \sh\left(\frac{\num(X)}{\num([0,1)}\right)\ =\
\overline{\mu}_L(X)$$
necessarily has translates $t+X$ with a different $\num_\beta$-measure:
$\num_\beta(t+X)\ne\num_\beta(X)$.
In fact, recall that Lebesgue measure $\mu_L$
is characterized as the unique translation-invariant measure
on the Borel subsets of $\R$ such that $\mu_L([0,1))=1$.\footnote
{~The family of Borel sets of a topological space
is the $\sigma$-algebra
generated by the open subsets.}
\end{remark}

\medskip
\subsection{Elementary numerosities and probability of infinite coin tosses}

The second application of our results on elementary numerosities
is about the existence of a non-Archimedean probability for infinite
sequences of coin tosses, which we propose as a sound mathematical model
for Laplace's original ideas. Recall the \emph{Kolmogorovian framework}:

\smallskip
\begin{itemize}
\item
The \emph{sample space}
$$\Omega\ =\ \{H,T\}^\N\ =\ \left\{\omega\mid\omega:\N\to\{H,T\}\right\}$$
is the set of sequences which take either $H$ (``head") or $T$ (``tail") as values.

\smallskip
\item
A \emph{cylinder set} of codimension $n$ is a set of the form:\footnote
{~We agree that $i_1<\ldots<i_n$.}
$$C_{(t_1,\ldots,t_n)}^{(i_1,\ldots,i_n)}\ =\
\left\{\omega\in\Omega\mid\omega(i_s)=t_s\ \text{ for }s=1,\ldots,n\right\}$$
\end{itemize}

\smallskip
From the probabilistic point of view, the cylinder set
$C_{(t_1,\ldots,t_n)}^{(i_1,\ldots,i_n)}$ represents the event that
for all $s=1,\ldots,n$, the $i_s$-th coin toss gives $t_s$ as outcome.
Notice that the family $\ring{C}$ of all cylinder sets
is a ring of sets over $\Omega$.

\smallskip
\begin{itemize}
\item
The function $\mu_C:\ring{C}\to[0,1]$ is defined by setting:
$$\mu_C\left(C_{(t_1,\ldots,t_n)}^{(i_1,\ldots,i_n)}\right)\ = \ 2^{-n}.$$
\end{itemize}

\smallskip
It is shown that $\mu_C$ is a pre-measure of probability on the ring $\ring{C}$.

\smallskip
\begin{itemize}
\item
$\A$ is the $\sigma$-algebra generated by the ring of
cylinder sets $\ring{C}$\,;

\smallskip
\item
$\mu:\A\to[0,1]$ is the unique probability measure that extends $\mu_C$,
as guaranteed by Caratheodory extension theorem.
\end{itemize}

\smallskip
The triple $(\Omega,\A,\mu)$ is named the \emph{Kolmogorovian probability}
for \emph{infinite sequences of coin tosses}.

\smallskip
In \cite{nap}, it is proved the existence of an elementary
numerosity $\num:\p(\Omega)\to\mathbb{F}$ which is coherent
with the pre-measure $\mu_C$. Namely,
by considering the ratio $P(E)=\num(E)/\num(\Omega)$
between the numerosity of the given event $E$ and
the numerosity of the whole space $\Omega$, then
one obtains a \emph{non-Archimedean} finitely additive probability
$P:\p(\Omega)\to[0,1]_\F$ that satisfies the following properties:

\smallskip
\begin{enumerate}
\item
If $F\subset\Omega$ is finite, then for all $E\subseteq\Omega$,
the conditional probability
$$P(E|F)\ =\ \frac{|E\cap F|}{|F|}.$$
\item
$P$ agrees with $\mu_C$ over all cylindrical sets:
$$P\left(C_{(t_1,\ldots,t_n)}^{(i_1,\ldots,i_n)}\right)\ =\
\mu_C\left(C_{(t_1,\ldots,t_n)}^{(i_1,\ldots,i_n)}\right)\ = \ 2^{-n}.$$
\end{enumerate}

\medskip
We are now able
to refine this result by showing that, up to infinitesimals,
we can take $P$ to agree with $\mu$ on the whole
$\sigma$-algebra $\A$.

\medskip
\begin{corollary}
Let $(\Omega,\A,\mu)$ be the Kolmogorovian probability
for infinite coin tosses. There exists an elementary numerosity
$\num:\p(\Omega)\rightarrow\F$ such that the corresponding
non-Archimedean probability $P(E)=\num(E)/\num(\Omega)$
satisfies the above properties $(1)$ and $(2)$, along with the
additional condition:

\smallskip
\begin{enumerate}
\item[(3)]
$\sh(P(E))=\mu(E)$ for all $E\in\A$.
\end{enumerate}
\end{corollary}

\begin{proof}
Recall that the family $\mathfrak{C}\subset\mathfrak{A}$ of
cylinder sets is a ring whose non-empty sets
have all positive measure. So, by applying Theorems
\ref{main theorem} and \ref{inner measure theorem},
we obtain an elementary numerosity $\num:\p(\Omega)\to\F$
such that for every positive number of the form
$\beta=\frac{\num(A^*)}{\mu(A^*)}$ one has:

\smallskip
\begin{enumerate}
\item[$(i)$]
$\num(C)=\num(C')$ whenever $C,C'\in\mathfrak{C}$ are
such that $\mu(C)=\mu(C')$\,;

\smallskip
\item[$(ii)$]
$\num_\beta(E)=\mu(E)$ for all $E\in\A$.

\end{enumerate}

\smallskip
Property $(1)$ trivially follows by recalling that elementary
numerosities of finite sets agree with cardinality:
$$P(E|F)\ =\ \frac{P(E\cap F)}{P(F)}\ =\
\frac{\frac{\num(E\cap F)}{\num(\Omega)}}{\frac{\num(F)}{\num(\Omega)}}\ =\
\frac{\num(E\cap F)}{\num(F)}\ =\ \frac{|E\cap F|}{|F|}.$$

Let us now turn to condition $(2)$.
Notice that for any fixed $n$-tuple of indices $(i_1,\ldots,i_n)$:

\begin{itemize}
\item
There are exactly $2^n$-many different $n$-tuples $(t_1, \ldots, t_n)$
of heads and tails;
\item
The associated cylinder sets $C_{(t_1,\ldots,t_n)}^{(i_1,\ldots,i_n)}$
are pairwise disjoint and their union equals the whole sample space $\Omega$.
\end{itemize}

By $(i)$, all those cylinder sets of codimension $n$
have the same numerosity
$\eta=\num\left(C_{(t_1,\ldots,t_n)}^{(i_1,\ldots,i_n)}\right)$
and so, by additivity, it must be $\num(\Omega)=2^n \cdot \eta$.
We conclude that
$$P\left(C_{(t_1, \ldots, t_k)}^{(i_1, \ldots, i_n)}\right)\ =\
\frac{\num\left(C_{(t_1, \ldots, t_k)}^{(i_1, \ldots, i_n)}\right)}
{\num(\Omega)}\ =\ \frac{\eta}{2^n\cdot\eta}\ =\ 2^{-n}.$$

We are left to prove $(3)$. By taking as
$\beta=\frac{\num(\Omega)}{\mu(\Omega)}=\num(\Omega)$,
property $(ii)$ ensures that for every $E\in\A$:
$$\mu(E)\ =\ \num_\beta(E)\ =\
\sh\left(\frac{\num(E)}{\beta}\right)\ =\
\sh\left(\frac{\num(E)}{\num(\Omega)}\right)\ =\
\sh(P(E)).$$
\end{proof}

\bigskip
\appendix
\section{Proof of Lemma \ref{combinatorial}}

Without loss of generality, we can assume that the given sets $A_i$ are
arranged in such a way that $A_1,\ldots,A_l\in\B$ and
$A_{l+1},\ldots,A_n\in\A\setminus\B$ for a suitable $l$.
It will be convenient in the sequel that the
considered elements in $\B$ be pairwise disjoint. To this end,
consider the partition $\{B_1,\ldots,B_h\}$ induced by $\{A_1,\ldots,A_l\}$, namely
$A_1\cup\ldots\cup A_l=B_1\sqcup\ldots\sqcup B_h$.\footnote
{~Recall that the \emph{partition induced} by a finite family $\{A_1,\ldots,A_n\}$ is
the partition on $A_1\cup\ldots\cup A_n$ given by the non-empty
intersections $\bigcap_{i=1}^n A_i^{\chi(i)}$ for
$\chi:\{1,\ldots,n\}\to\{-1,1\}$,
where $A_j^{1}=A_j$ and $A_j^{-1}=(\bigcup_{i=1}^n A_i)\setminus A_j$.}
(Notice that,
by the ring properties of $\B$, every piece $B_s$ belongs to $\B$.)
Finally, let
$$\bigcup_{i=1}^n A_i\ =\ C_1\sqcup\ldots\sqcup C_p \sqcup D_1\sqcup\ldots\sqcup D_q$$
be the partition induced by $\{B_1,\ldots,B_h,A_{l+1},\ldots,A_n \}$,
where $\mu(C_s)>0$ for $s=1,\ldots,p$ and $\mu(D_t)=0$ for $t=1,\ldots,q$.
For every $s=1,\ldots,h$, the set $B_s$ include at least one
piece $C_j$ of positive measure in the above partition.
Moreover, since $B_1,\ldots,B_h$ are pairwise disjoint, by re-arranging if necessary,
we can also assume that $C_s\subseteq B_s$ for $s=1,\ldots,h$.

\smallskip
Now recall the following \emph{Dirichlet's simultaneous approximation theorem}
(see \emph{e.g.} \cite{hw} \S 11.12):
``Given finitely many real numbers $y_s>0$, for every $\varepsilon>0$
there exist arbitrarily large numbers $N \in \N$ such that every
fractional part $\{N\cdot y_s\}=N\cdot y_s -[N\cdot y_s]<\varepsilon$''.
So, if we let

\smallskip
\begin{itemize}
\item
$\alpha=\mu\left(\bigcup_{i=1}^n A_i\right)$

\smallskip
\item
$c=\min\{\mu(C_s)\mid s=1,\ldots, p\}$
\end{itemize}

\noindent
then we can pick a natural number $N$ such that:

\smallskip
\begin{enumerate}
\item[$(a)$]
$N> \frac{\alpha\,(2m+1)\,(k+1)}{c^2}$\,;


\smallskip
\item[$(b)$]
$e_s=\{N\cdot\mu(C_s)\}<\frac{1}{p}$ for all $s=1,\ldots,p$\,.
\end{enumerate}

\smallskip
Denote by

\smallskip
\begin{itemize}
\item
$C=\bigsqcup_{s=1}^p C_s$ the ``positive part" of the partition\,;

\smallskip
\item
$D=\bigsqcup_{t=1}^q D_t$ the ``negligible part" of the partition\,;

\smallskip
\item
$F=\{x_1,\ldots,x_k \}$.
\end{itemize}

\smallskip
Then, set

\smallskip
\begin{itemize}
\item
$N_s=[N\cdot\mu(C_s)]$ for $s=1,\ldots,p$\,;

\smallskip
\item
$M_s=|B_s\cap D\cap F|$ for $s=1,\ldots,h$\,.
\end{itemize}

\smallskip
Notice that $N_s>k$ for all $s$. In fact,
by the above conditions $(a)$ and $(b)$:
\begin{eqnarray*}
\nonumber
N_s & = & N\cdot\mu(C_s)-e_s\  >\
\frac{\alpha\,(2m+1)\,(k+1)}{c^2}\cdot\mu(C_s)-e_s
\\
\nonumber
{} & > &
\frac{\alpha\cdot\mu(C_s)}{c^2}\cdot(k+1) - e_s\ >\
1\cdot(k+1)-1\ = \ k.
\end{eqnarray*}

\smallskip
For $s=1,\ldots,h$, pick a finite subset $\lambda_s\subset C_s$ containing exactly
$(N_s - M_s)$-many elements, and such that $C_s\cap F\subseteq\lambda_s$.
Observe that this is in fact possible because
$$|C_s\cap F|\ \le\ |B_s\cap C\cap F|\ =\ |B_s\cap F|-M_s\ \le\ k-M_s\ <\ N_s-M_s.$$
For $s = h+1, \ldots, p$, pick a finite subset
$\lambda_s \subset C_s$ containing exactly $N_s$-many elements.
Finally, define
$$\lambda\ =\ F\cup\bigcup_{s=1}^p \lambda_s.$$

We claim that $\lambda$ has the desired properties.
Condition $(1)$ is trivially is satisfied because $F\subseteq\lambda$ by definition.
For every $i=1,\ldots,n$ let:
$$G(i)\ =\ \{s\leq h\mid C_s\subseteq A_i\}\quad\text{and}\quad
G'(i)\ =\ \{s>h\mid C_s\subseteq A_i\}.$$

With the above definitions, we obtain:
\begin{eqnarray*}
|\lambda\cap A_i| & = &
\sum_{s\in G(i)}|\lambda_s| + \sum_{s\in G'(i)}|\lambda_s|\ +\ |A_i\cap D\cap F|
\\
& = & \sum_{s\in G(i)}(N_s - M_s)\ +
\sum_{s\in G'(i)} N_s\ + \ |A_i\cap D\cap F|
\\
& = & \sum_{s\in G(i)\cup G'(i)}\!\!\!\!\!N_s\ - \sum_{s\in G(i)} M_s\ +\ |A_i\cap D\cap F|
\\
& = & N\cdot\left(\sum_{s\in G(i)\cup G'(i)}\!\!\!\mu(C_s)\right) -
\varepsilon_i - \eta_i + \vartheta_i
\\
& = & N\cdot\mu(A_i) - \varepsilon_i - \eta_i + \vartheta_i
\end{eqnarray*}

where:

\smallskip
\begin{itemize}
\item
$\varepsilon_i=\sum_{s\in G(i)\cup G'(i)}e_s\ \leq\ \sum_{s=1}^{p} e_s\  <\  1$
by condition $(b)$\,;

\smallskip
\item
$\eta_i=\sum_{s \in G(i)} M_s\ \leq\ \sum_{s=1}^h |B_s\cap D\cap F| \ \leq\  |F| = k$\,;

\smallskip
\item
$\vartheta_i = |A_i \cap D \cap F|\ \leq\ k$.
\end{itemize}

\smallskip
If $A_i\in\B$, \emph{i.e.} if $i\leq l$, recall that $A_i=\bigsqcup_{s\in S(i)}B_s$ for a suitable
$S(i)\subseteq\{1,\ldots,h \}$.
Since $C_s\subseteq B_s$ for all $s=1,\ldots,h$, it must be $G(i)=S(i)$.
So, for $i\le l$ we have
\begin{eqnarray*}
\nonumber
\eta_i & = &
\sum_{s\in S(i)} M_s\ =\
\sum_{s\in S(i)} |B_s\cap D\cap F|\ =\
\left|\left(\bigsqcup_{s\in S(i)} B_s\right)\cap D\cap F\right|
\\
\nonumber
{} & = &
|A_i\cap D\cap F|\ =\ \vartheta_i,
\end{eqnarray*}
and hence $|\lambda\cap A_i|=N\cdot\mu(A_i)-\varepsilon_i$.
In consequence, for every $i,j\le l$ such that $\mu(A_i) = \mu(A_j)$, one has that
$$\big||\lambda\cap A_i| - |\lambda\cap A_j| \big|\ =\
|N\cdot\mu(A_i) - \varepsilon_i - N\cdot\mu(A_j) + \varepsilon_j|\ =\
|\varepsilon_j - \varepsilon_i|.$$
Now notice that $|\varepsilon_j-\varepsilon_i|\leq
\max\{\varepsilon_i, \varepsilon_j\}<1$, and so
the natural numbers $|\lambda\cap A_i|=|\lambda\cap A_j|$ necessarily coincide.
This completes the proof of $(2)$.

\smallskip
As for $(3)$, notice that $|\lambda\cap A_i|=N\cdot\mu(A_i)+\zeta_i$
where $\zeta_i=(\vartheta_i-\eta_i)-\varepsilon_i$ is such that
$-(k+1)<\zeta_i\le k$. For every $i,j$ such that $\mu(A_j)\ne 0$, we have that
$$\frac{N\cdot\mu(A_i)+\zeta_i}{N\cdot\mu(A_j)+\zeta_j}-\frac{\mu(A_i)}{\mu(A_j)}\ =\
\frac{\mu(A_j)\cdot\zeta_i-\mu(A_i)\cdot\zeta_j}{N\cdot\mu(A_j)^2+\mu(A_j)\cdot\zeta_j}.$$
Now, the absolute value of the numerator
$$|\mu(A_j)\cdot\zeta_i-\mu(A_i)\cdot\zeta_j|\ <\
(\mu(A_i)+\mu(A_j))\cdot(k+1)\ \le\ 2\,\alpha\,(k+1)\,;$$
and the absolute value of the denominator
\begin{eqnarray*}
\nonumber
|N\cdot\mu(A_j)^2+\mu(A_j)\cdot\zeta_j| & > &
N\,c^2 - \alpha\,(k+1)
\\
\nonumber
{} & \hspace{-5cm} \ge &
\hspace{-2.5cm}
\alpha\,(2m+1)\,(k+1)-\alpha\,(k+1)\ =\ 2m\,\alpha\,(k+1).
\end{eqnarray*}
So, we reach the thesis:
$$\left|\frac{|\lambda\cap A_i|}{|\lambda\cap A_j|} -
\frac{\mu(A_i)}{\mu(A_j)}\right|\ <\
\frac{2\,\alpha\,(k+1)}{2m\,\alpha\,(k+1)}\ =\ \frac{1}{m}.$$

\end{document}